\documentclass[10pt,reqno]{article}

\usepackage{amsthm,amsmath,amssymb,amsbsy,amsfonts,graphics,epsfig,graphicx,mathdots,cleveref}

\newcommand{\C}{{\mathbb C}}

\newcommand{\N}{{\mathbb N}}

\newcommand{\beq}{\begin{equation}}
	\newcommand{\eeq}{\end{equation}}

\newtheorem{theo}{{\bf Theorem}}[section]
\newtheorem{cor}[theo]{{\bf Corollary}}
\newtheorem{lem}[theo]{{\bf Lemma}}
\newtheorem{Prop}[theo]{{\bf Proposition}}

\newcounter{definition}
\newenvironment{defn}[1][]{\refstepcounter{definition}\par\medskip
	\noindent \textbf{Definition~\thedefinition. #1} \rmfamily}{\medskip}

\newcounter{example}
\newenvironment{ex}[1][]{\refstepcounter{example}\par\medskip
	\noindent \textbf{Example~\theexample. #1} \rmfamily}{\medskip}

\newcounter{remark}   
\newenvironment{rem}[1][]{\refstepcounter{remark}\par\medskip
	\noindent \textbf{Remark~\theremark. #1} \rmfamily}{\medskip}

\numberwithin{equation}{section}

\author{A. Bourget and T. McMillen \footnote{\textbf{Mailing address}: Department of Mathematics, California State University (Fullerton),
		McCarthy Hall 154, Fullerton CA 92834 (US). \textbf{Email address}: abourget@fullerton.edu, \mbox{tmcmillen@fullerton.edu}} \\ Department of Mathematics \\ California State University, Fullerton}

\title{Stability of Multidimensional Linear Discrete Systems}

\date{ }

\begin{document}
	
\maketitle	

\begin{abstract}
	We study linear multidimensional discrete systems of the form
\[
x_\alpha=\sum_{k=1}^m A_kx_{\alpha-e_k},
\qquad \alpha\in\mathbb N^m,
\]
with uniformly bounded boundary data prescribed on the codimension-one coordinate faces and attractivity understood in the sense of Pringsheim convergence. We first derive an explicit solution representation in terms of sums of ordered matrix products and introduce a shell decomposition whose associated sums satisfy a one-dimensional affine recursion. Our main result is a generalized Fornasini--Marchesini theorem in arbitrary dimension: asymptotic stability is equivalent to the condition
\[ \rho\!\left(z_1A_1 + \cdots + z_m A_m \right)<1 \qquad \text{for every } |z_1| = \cdots = |z_m|=1. \]
 We then study the critical spectral boundary, where the maximal toral spectral radius equals one. We prove that a nonsemisimple unimodular eigenvalue of a toral pencil forces instability, yielding necessary spectral conditions for Lyapunov stability. In the simultaneously diagonalizable case, we obtain a complete classification of asymptotically stable, neutrally stable, and unstable systems in terms of the common eigenvalues. Finally, we show that bounds involving only the individual spectral radii can fail outside this setting: even when all coefficient matrices are nilpotent and their spectral radii sum to zero, a critical system may still be unstable.
\end{abstract}

\medskip
\noindent\textbf{Keywords:} Multidimensional discrete systems; Fornasini--Marchesini models; Asymptotic stability; Lyapunov stability.

\section{Introduction}

Multidimensional linear discrete dynamical systems, or \emph{$m$D systems}, arise naturally when the evolution of a state depends on several independent discrete
directions. Examples occur in multidimensional filtering and signal processing, lattice models, iterative schemes on grids, and numerical discretizations of partial differential equations involving several spatial or spatio-temporal indices; see, for example, \cite{Bo82,Ga01,Ka85,ZhYa18}. Stability and stabilization of linear multidimensional discrete systems have also been studied extensively in the frequency domain; see, for instance, \cite{LiXuLin2013}. In such settings, the state is indexed by a multi-index $\alpha\in\mathbb N^m$, and its evolution depends on neighboring states in the coordinate directions.

A classical framework in dimension two is provided by the Fornasini--Mar\-chesini (FM) models \cite{FoMa78,FoMa80}, which play a central role in $2$D systems theory; see also the monographs \cite{Bo82,Ka85}, the survey \cite{TiDh12}, and the retrospective account \cite{GaRo18}. Stability criteria for two-dimensional discrete systems have been developed from several points of view, including spectral and generalized-eigenvalue methods \cite{FuChNi06}, behavioral approaches \cite{Valcher2000}, and frequency-domain techniques \cite{LiXuLin2013}. More recent developments include Lyapunov and LMI methods for multidimensional state-space models \cite{ElosmaniEtAl2021,KanellakisEtAl2024}. In its simplest homogeneous form, the FM recurrence is
\[ x_{i,j} = A_1x_{i-1,j}+A_2x_{i,j-1}, \qquad (i,j)\in\mathbb N^2, \]
with boundary data prescribed on the coordinate axes. A fundamental result in the $2$D theory characterizes asymptotic stability in terms of the matrix
pencil
\[ z_1A_1+z_2A_2: \]
asymptotic stability is equivalent to a strict spectral-radius condition on the unit torus, or, equivalently, to the absence of zeros of the associated determinant on the closed bidisk. The stability of this doubly indexed system under bounded boundary data was studied in detail in our earlier work \cite{BoMc26}, which provides the two-dimensional starting point for several of the questions considered here.

In this paper we study the natural $m$-dimensional extension of the homogeneous FM recurrence. Let $A_1,\ldots,A_m\in\mathcal M_n(\mathbb C)$, with $m\geq2$, and consider
\begin{equation*}
	x_\alpha = \sum_{k=1}^m A_kx_{\alpha-e_k}, \qquad \alpha\in\mathbb N^m.
\end{equation*}
Boundary data are prescribed on the open coordinate faces of codimension one,
\[ \mathcal H = \left\{ \beta\in\mathbb N_0^m: \beta\text{ has exactly one zero component} \right\}, \]
and are assumed to belong to $\ell^\infty(\mathcal H;\mathbb C^n)$. Attractivity is understood in the sense of Pringsheim convergence: $x_\alpha\to0$ as all components of
$\alpha$ tend to infinity.

The choice of admissible boundary data and of the asymptotic regime is important in multidimensional systems. In several formulations in the literature, attractivity is considered for boundary data that vanish at infinity, typically in a $c_0$-type space, together with convergence of the whole array as $|\alpha|\to\infty$; see, for example,
\cite{Ba18,Ba23,BaCl23,David2018}. A different approach, based on discrete behaviors with a distinguished time direction and spectral conditions on characteristic varieties, was developed by Oberst and Scheicher \cite{OberstScheicher2014}. The framework adopted here is different from both of these settings. We allow arbitrary bounded boundary data,
\[ c_0(\mathcal H;\mathbb C^n) \subsetneq \ell^\infty(\mathcal H;\mathbb C^n), \]
but require decay only in the Pringsheim sense. Thus our class of admissible boundary data is larger, while the asymptotic requirement is weaker. These two notions of attractivity are therefore not directly comparable. The precise boundary convention and convergence notion used throughout the paper are described in Section~2.

\medskip
\noindent
\textbf{Representation formulas and a generalized Fornasini--Marchesini theorem.}
Our first step is to develop tools that make the multidimensional recurrence more tractable. For each multi-index $\beta\in\mathbb N_0^m$, we introduce a matrix $W_\beta$ obtained by summing all ordered products, or \emph{words}, in $A_1,\ldots,A_m$ in which $A_k$ occurs exactly $\beta_k$ times. Related transition sequences occur in the two-dimensional literature; in particular, David \cite{David2018} introduces a doubly indexed sequence corresponding to the $2$D version of these word
sums. Proposition~\ref{bigform} gives an explicit representation of $x_\alpha$ in terms of the matrices $W_\beta$ and the boundary values on the
codimension-one faces.

A second useful representation is obtained by decomposing the lattice into $\ell^1$-level sets, or shells. In Proposition~\ref{ym} we show that the corresponding shell sums
\[ y_p=\sum_{\alpha\in S_p}x_\alpha \]
satisfy a one-dimensional affine recurrence whose linear part is governed by the single matrix
\[ A_1+\cdots+A_m. \]
This reduction provides a link between the multidimensional recurrence and one-dimensional spectral behavior, and it plays an important role in the instability results developed later.

Our principal result is Theorem~\ref{bigTH}, a generalized Fornasini--Marchesini theorem for the present $m$D setting. We prove that $\mathcal S(A_1,\ldots,A_m)$ is asymptotically stable if and
only if
\[ \rho(z_1A_1 + \cdots + z_m A_m)<1 \qquad \text{for every } |z_1|= \cdots =|z_m|=1.
\]
Equivalently,
\[ \rho(z_1A_1 + \cdots + z_m A_m)<1 \qquad \text{for every } |z_1| \leq 1, \ldots , |z_m|\leq 1\]
or
\[ \det(I-(z_1A_1 + \cdots+ z_m A_m))\neq0 \qquad \text{for every } |z_1| \leq 1, \ldots , |z_m|\leq 1. \]
Thus asymptotic stability is completely characterized by the multivariate pencil $z_1A_1 + \cdots + z_m A_m$. Several practical sufficient conditions follow immediately, including criteria involving induced matrix norms and the entrywise absolute values of the coefficient matrices.

\medskip
\noindent
\textbf{The critical spectral boundary.}
The generalized FM theorem completely determines the stability behavior away from the spectral boundary. If
\[
\max_{|z_1|= \cdots=|z_m|=1}\rho(z_1A_1 + \cdots + z_m A_m)<1,
\]
the system is asymptotically stable, whereas if $\rho(z_1A_1+\cdots+z_mA_m)>1$ for some $|z_1|=\cdots=|z_m|=1$, the system is unstable. The
remaining case is therefore the critical regime
\[
\max_{|z_1|=\cdots=|z_m|=1}\rho(z_1A_1+\cdots+z_mA_m)=1.
\]
Here the spectral radius alone no longer determines Lyapunov stability.

Our main result in this regime shows that the Jordan structure of the critical pencils provides an additional obstruction. In Theorem~\ref{defective-pencil} we prove that if, for some $|z_1|=\cdots=|z_m|=1$, the matrix $z_1A_1+ \cdots + z_m A_m$ has a nonsemisimple eigenvalue of modulus one, then $\mathcal S(A_1,\ldots,A_m)$ is unstable. Consequently, Lyapunov stability requires
\[ \rho(z_1 A_1+ \cdots + z_mA_m) \leq 1 \qquad \text{for every } |z_1|= \cdots = |z_m|=1, \]
and every unimodular eigenvalue of every toral pencil must be semisimple. These conditions are not sufficient in general, however: even for commuting coefficient matrices there are systems satisfying both conditions that are nevertheless unstable. Thus, the spectral-radius bound together with semisimplicity of the unimodular eigenvalues of the toral pencils is not sufficient to characterize Lyapunov stability in the critical regime.

The examples in Section~5 illustrate both possibilities on the spectral boundary. A critical system may be neutrally stable, while the presence of a defective unimodular eigenvalue can produce instability. This contrast highlights the qualitative difference between the strict stability region and its boundary.

\medskip
\noindent
\textbf{The simultaneously diagonalizable case.}
A complete classification is possible when $A_1,\ldots,A_m$ are simultaneously diagonalizable. If
\[ A_k=X\Lambda_kX^{-1}, \qquad \Lambda_k = \operatorname{diag}(a_{k,1},\ldots,a_{k,n}), \]
then the system decouples along the common eigendirections. In particular,
\[ \max_{|z_1|=\cdots=|z_m|=1}\rho(z_1 A_1 + \cdots + z_m A_m) = \max_{1\leq j\leq n} \sum_{k=1}^m|a_{k,j}|. \]
Theorem~\ref{simultaneously-diagonalizable} therefore reduces the stability problem to scalar inequalities: the system is asymptotically stable if
\[ \sum_{k=1}^m|a_{k,j}|<1 \qquad \text{for every }j, \]
neutrally stable if all of these sums are at most one and equality occurs for at least one common eigendirection, and unstable if one of the sums exceeds one.

This characterization yields a simple sufficient condition in terms of the individual spectral radii. For simultaneously diagonalizable coefficient matrices,
\[ \sum_{k=1}^m\rho(A_k)<1 \]
implies asymptotic stability, while
\[ \sum_{k=1}^m\rho(A_k)\leq1 \]
implies Lyapunov stability. The assumption of simultaneous diagonalizability is essential. We exhibit a critical system for which every coefficient matrix is nilpotent, so that
\[ \sum_{k=1}^m\rho(A_k)=0<1, \]
but the system is nevertheless unstable. In that example the instability is caused by a defective unimodular eigenvalue of a critical pencil. Thus individual spectral radii alone do not control stability for general matrix families, even under a strict inequality.

Some of the ideas developed here originate in our two-dimensional analysis \cite{BoMc26}. The present work extends the bounded-boundary framework to arbitrary dimension, develops the corresponding word and shell representations, establishes the generalized FM characterization, and identifies new obstructions that arise on the critical spectral boundary.

\medskip
\noindent
\textbf{Organization.}
Section~2 introduces the notation, boundary convention, stability notions, and a basic unimodular invariance property. Section~3 develops the word-sum representation and the shell decomposition and derives a general instability criterion. Section~4 proves the generalized Fornasini--Marchesini theorem and records several consequences. Section~5 studies the critical spectral boundary, including the instability criterion for nonsemisimple unimodular eigenvalues and examples of both neutral stability and instability. Section~6 gives a complete classification for simultaneously diagonalizable coefficient families and shows that spectral-radius conditions involving the individual
matrices need not extend beyond that setting. Section~7 concludes the paper.

\section{Definitions and notation}

We denote by
\[ \mathbb D:=\{z\in\mathbb C:|z|<1\}, \qquad \overline{\mathbb D}:=\{z\in\mathbb C:|z|\leq1\}, \qquad \mathbb T:=\{z\in\mathbb C:|z|=1\}. \]
Their $m$-fold Cartesian products are denoted by $\mathbb D^m,\  \overline{\mathbb D}^{\,m},\ \mathbb T^m$, respectively.

We employ standard multi-index notation: multi-indices are denoted by Greek letters $\alpha,\beta,\ldots$, where
\begin{equation*}
\alpha=(\alpha_{1},\ldots,\alpha_{m})\in \mathbb N^{m}\ \text{or}\ \mathbb N_{0}^{m}.
\end{equation*}
Here, $\N$ denotes the set of positive integers and $\N_0$ the set of all non-negative integers. For every multi-index $\alpha=(\alpha_1,\ldots,\alpha_m)$ and
$z=(z_1,\ldots,z_m)\in\C^m$, we define
\[ |\alpha|:=\alpha_1+\cdots+\alpha_m, \qquad z^\alpha:=z_1^{\alpha_1}\cdots z_m^{\alpha_m}. \]
We say that $\alpha$ is a multi-index of order $p$ if $|\alpha|=p$.  For $k\in\{1,\ldots,m\}$, let $e_{k}\in\mathbb N_{0}^{m}$ denote the $k$-th canonical multi-index, i.e., the element with a $1$ in the $k$-th component and zeros elsewhere. We further set
\begin{equation} \label{eq:Zdef}
	\mathcal H:=\{\beta\in\mathbb N_{0}^{m}:\ \beta \text{ has exactly one zero component}\}.
\end{equation}

For example, when $m=2$,
\[ \mathcal H = \{(i,0):i\in\mathbb N\} \cup \{(0,j):j\in\mathbb N\}, \]
so that $\mathcal H$ consists of the two coordinate axes, with the origin excluded. Thus our boundary convention reduces in this case to the usual boundary data for the two-dimensional Fornasini--Marchesini model. When $m=3$,
\[ \mathcal H = \{(0,j,k):j,k\in\mathbb N\} \cup \{(i,0,k):i,k\in\mathbb N\} \cup \{(i,j,0):i,j\in\mathbb N\}, \]
so that the boundary data are prescribed on the three open coordinate faces of $\mathbb N_0^3$.

\medskip

Let $A_{1},\ldots,A_{m}\in \mathcal M_{n}(\mathbb C)$. We write $\mathcal S(A_{1},\ldots,A_{m})$ for the $m$-dimensional discrete dynamical system defined by the recurrence
\begin{equation} \label{rec rel}
	x_{\alpha}=\sum_{k=1}^{m}A_{k}\,x_{\alpha-e_{k}},\qquad \alpha\in\mathbb N^{m},
\end{equation}
together with uniformly bounded boundary data prescribed on the open coordinate hyperplanes (faces) of codimension $1$, namely,
\begin{equation} \label{bound cond}
 (x_\beta)_{\beta \in \mathcal H} \in \ell^{\infty}(\mathcal H;\mathbb C^{n}).
\end{equation}
For boundary data $x=(x_\beta)_{\beta \in \mathcal H}$, set
\begin{equation}
		\| x \|_{\infty,\mathcal H} := \sup_{\beta \in \mathcal H} \| x_\beta\|.
\end{equation}
Unless stated otherwise, any stability property of $\mathcal S(A_{1},\ldots,A_{m})$ refers to the equilibrium $x\equiv 0$. The notions of stability introduced below are adapted from the standard definitions in Strogatz \cite{st94}. 

For $A\in\mathcal M_n(\mathbb C)$, we denote its spectrum by $\sigma(A)$ and its spectral radius by
\[ \rho(A):=\max\{|\lambda|:\lambda\in\sigma(A)\}. \]
Associated with the coefficient matrices $A_1,\ldots,A_m$ is the matrix pencil
\[ P(z):=\sum_{k=1}^m z_kA_k, \qquad z=(z_1,\ldots,z_m)\in\mathbb C^m. \]

\begin{defn} \label{def:stability}
	The system $\mathcal S(A_{1},\ldots,A_{m})$ is said to be:
	\begin{enumerate}
		\item[(a)] \emph{Lyapunov stable} if for every $\varepsilon>0$ there exists $\delta>0$ such that
		\begin{equation*}
		\|x_{\alpha}\|<\varepsilon \quad \text{for all } \alpha\in\mathbb N^{m},
		\end{equation*}
		whenever $\|x \|_{\infty, \mathcal H}<\delta$.
		\item[(b)] \emph{Attractive} if there exists $\delta>0$ such that
		\begin{equation*}
		\lim_{\alpha\to\infty} x_{\alpha}=0
		\end{equation*}
		whenever $\|x \|_{\infty, \mathcal H}<\delta$.
		\item[(c)] \emph{Asymptotically stable} if it is Lyapunov stable and attractive.
		\item[(d)] \emph{Neutrally stable} if it is Lyapunov stable but not attractive.
		\item[(e)] \emph{Unstable} if it is not Lyapunov stable.
	\end{enumerate}
\end{defn}

\begin{rem} \label{boundary-remark}
	We record several comments concerning Definition~1 and the framework adopted throughout the paper.
	\begin{enumerate}
		\item The norm $\|\cdot\|$ in Definition~1 is an arbitrary norm on $\mathbb C^n$. Since all norms on $\mathbb C^n$ are equivalent, the stability notions above do not depend on this choice. When convenient, we shall use the max norm
		\[ \|x\|_\infty=\max\{|x_k|:k=1,\ldots,n\}. \]
		All matrix norms are assumed to be induced and hence submultiplicative.
		
		\item The convergence $x_\alpha\to0$ in the definition of attractivity is understood in the sense of Pringsheim: for every $\varepsilon>0$ there exists $N\in\mathbb N$ such that
		\[ \|x_\alpha\|<\varepsilon \qquad \text{whenever} \qquad \alpha_k\geq N,\quad k=1,\ldots,m. \]
		Thus this convergence concerns the interior asymptotic regime. For example, it does not require
		\[ x_{(j,1,\ldots,1)} \longrightarrow0 \qquad\text{as }j\to\infty, \]
		since the last $m-1$ coordinates remain fixed.
		
		\item The stability notions considered here are properties of the system relative to the class of admissible boundary data and the prescribed mode of convergence at infinity. In the present paper the boundary data are uniformly bounded but are not required to tend to zero along the coordinate faces, while attractivity is formulated using Pringsheim convergence. As discussed in the Introduction, this framework differs from $c_0$-type formulations in which the boundary data vanish at infinity and convergence of the whole array is required.
		
		\item Fix $k\in\{1,\ldots,m\}$. The one-dimensional dynamics generated by $A_k$ are contained in the multidimensional recurrence as a special case. Indeed, let
		\[ \mathbf 1:=(1,\ldots,1)\]
		and prescribe boundary data by
		\[ x_{\mathbf 1-e_k}=u, \qquad x_\beta=0 \quad \text{for every } \beta \in \mathcal H,\ \beta\neq\mathbf 1-e_k,\]
		where $u\in\mathbb C^n$. Consider the sequence of interior multi-indices
		\[ \alpha^{(j)} := \mathbf 1+(j-1)e_k, \qquad j\in\mathbb N.\]
		For $j=1$, the recurrence gives
		\[ x_{\alpha^{(1)}} = x_{\mathbf 1} = A_ku,\]
		since all the other boundary predecessors of $\mathbf 1$ have been assigned the value zero. More generally, for $j\geq2$, the predecessors $\alpha^{(j)}-e_\ell$, with $\ell\neq k$, belong to $\mathcal H$ and have boundary value zero. Hence
		\[ x_{\alpha^{(j)}} = A_kx_{\alpha^{(j-1)}}. \]
		It follows inductively that
		\[ x_{\alpha^{(j)}}=A_k^{\,j}u, \qquad j\in\mathbb N. \]
		Thus the dynamics of the one-dimensional system
		\[ u_{j+1}=A_ku_j \]
		occur along the $k$-th coordinate direction of the multidimensional system. In particular, if this one-dimensional system is unstable, then $\mathcal S(A_1,\ldots,A_m)$ is unstable.
\end{enumerate}		
\end{rem}

We conclude this section by recording a simple invariance property that will be used repeatedly below.

\begin{rem}\label{rotation-invariance}
The stability type of $\mathcal S(A_1,\ldots,A_m)$ is invariant under unimodular rotations of the coefficient matrices. More precisely, let
\[ z=(z_1,\ldots,z_m)\in\mathbb T^m \]
and define
\[ A_k':=z_kA_k, \qquad k=1,\ldots,m. \]
Then
\[ \mathcal S(A_1,\ldots,A_m) \quad \text{and} \quad \mathcal S(A_1',\ldots,A_m')\]
have the same stability type. Indeed, let $(x_\alpha)$ be a solution of $\mathcal S(A_1,\ldots,A_m)$ and define
\[ x_\alpha':=z^\alpha x_\alpha, \qquad z^\alpha=z_1^{\alpha_1}\cdots z_m^{\alpha_m}.\]
Then
\begin{align*}
	x_\alpha' &= z^\alpha \sum_{k=1}^m A_kx_{\alpha-e_k} 
	= \sum_{k=1}^m z_kA_k\,z^{\alpha-e_k}x_{\alpha-e_k} 
	= \sum_{k=1}^m A_k'x_{\alpha-e_k}'.
\end{align*}
Thus $(x_\alpha')$ is a solution of $\mathcal S(A_1',\ldots,A_m')$ with boundary data
\[ (x_\beta')_{\beta\in\mathcal H} = (z^\beta x_\beta)_{\beta\in\mathcal H}. \]
Since $z\in\mathbb T^m$, we have $|z^\alpha|=1$ for every $\alpha\in\mathbb N_0^m$, and therefore
\[ \|x_\alpha'\|=\|x_\alpha\|, \qquad \|x_\beta'\|=\|x_\beta\|. \]
Moreover,
\[ x_\alpha\longrightarrow0 \quad\text{in the sense of Pringsheim} \]
if and only if
\[ x_\alpha' \longrightarrow 0 \quad \text{in the sense of Pringsheim}.\]
Hence Lyapunov stability, attractivity, asymptotic stability, neutral stability, and instability are all preserved under the replacement
\[ A_k\longmapsto z_kA_k, \qquad k=1,\ldots,m.\]
\end{rem}

\section{Preliminary results} 

Before we start the stability analysis, it is useful to have an explicit ``unrolled'' representation of the solution of the multi-index recurrence \eqref{rec rel} in terms of the boundary data. The idea is to trace a value $x_\alpha$ backwards to the boundary by repeatedly applying the recursion, keeping track of how many times each matrix $A_k$ is used along a path. This naturally leads to sums of ordered products of the matrices---i.e., ``words'' in the alphabet $\{A_1,\dots,A_m\}$---grouped according to their multiplicities. 

The next proposition makes this combinatorial structure precise by introducing the regions $R_i(\alpha)$ that capture the last step taken in the $i$-th direction and the quantities $W_\beta$, which collect all words in which $A_k$ appears exactly $\beta_k$ times. In the two-dimensional case, analogous word sums and solution representations were used in~\cite{BoMc26}. The formulation below extends this construction to arbitrary dimension and to boundary data prescribed on the codimension-one faces.

For $\beta=(\beta_1,\ldots,\beta_m)\in\mathbb N_0^m$, let
\[ \mathcal W(\beta) := \left\{ (i_1,\ldots,i_{|\beta|})\in\{1,\ldots,m\}^{|\beta|}: \#\{r:i_r=k\}=\beta_k,\quad k=1,\ldots,m \right\}.\]
We define
\[ W_\beta := \sum_{(i_1,\ldots,i_{|\beta|})\in\mathcal W(\beta)} A_{i_1}\cdots A_{i_{|\beta|}}, \]
with the convention $W_0=I$. We also set $W_\gamma=0$ whenever $\gamma\in\mathbb Z^m$ has at least one negative component.
	
Thus $W_\beta$ is the sum of all words in $A_1,\ldots,A_m$ in which $A_k$ occurs exactly $\beta_k$ times. It may be useful to illustrate the word sums $W_\alpha$ in a few simple cases. For $i\neq j$, we have
\[ W_{e_i}=A_i,\]
and
\[ W_{e_i+e_j}=A_iA_j+A_jA_i.\]
More generally, in the two-dimensional case $m=2$,
\[ W_{(2,1)}(A_1,A_2) = A_1^2A_2+A_1A_2A_1+A_2A_1^2,\]
since these are precisely the three words containing two copies of $A_1$ and one copy of $A_2$. Thus $W_\alpha$ collects all ordered products of the matrices $A_1,\ldots,A_m$ having the multiplicities prescribed by the multi-index $\alpha$.
	
Partitioning these words according to their first or last letter gives the identities
	\begin{equation}\label{word-recursion}
		W_\beta = \sum_{k=1}^m A_kW_{\beta-e_k} =\sum_{k=1}^m W_{\beta-e_k}A_k, \qquad \beta \neq 0.
	\end{equation}

For $\alpha\in\mathbb N^m$ and $i\in\{1,\ldots,m\}$, define
\begin{equation}\label{Ri-def}
	\mathcal R_i(\alpha) := \left\{ \beta\in\mathbb N_0^m: \beta_i=\alpha_i-1,\quad 0 \leq \beta_k\leq\alpha_k-1 \text{ for }k\neq i \right\}.
\end{equation}

\begin{Prop}\label{bigform}
For every $\alpha\in\mathbb N^m$, the solution of the recurrence \eqref{rec rel} satisfies
	\begin{equation}\label{x alpha}
		x_\alpha = \sum_{i=1}^m \sum_{\beta\in\mathcal R_i(\alpha)} W_\beta A_i x_{\alpha-\beta-e_i}.
	\end{equation}
\end{Prop}

\begin{proof}
We argue by induction on
\[ d(\alpha):=|\alpha|-m =\sum_{k=1}^m(\alpha_k-1), \]
which measures the distance of $\alpha\in\mathbb N^m$ from $\mathbf 1=(1,\ldots,1)$. If $d(\alpha)=0$, then $\alpha=\mathbf 1$. 

For each $i\in\{1,\ldots,m\}$, we have
\[ \mathcal R_i(\mathbf 1)=\{0\}. \]
Hence, since $W_0=I$,
\[ \sum_{i=1}^m \sum_{\beta\in\mathcal R_i(\mathbf 1)} W_\beta A_i x_{\mathbf 1-\beta-e_i} = \sum_{i=1}^m A_i x_{\mathbf 1-e_i} =x_{\mathbf 1}, \]
by the defining recurrence. Thus the formula holds when $d(\alpha)=0$.
	
Now let $d(\alpha)>0$ and assume that the formula holds for every $\alpha'\in\mathbb N^m$ with $d(\alpha')<d(\alpha)$. Set
\[ J_0:=\{j:\alpha_j=1\}, \qquad J_+:=\{j:\alpha_j\geq2\}. \]
Using the recurrence, we write
\begin{equation}\label{prop31-split}
	x_\alpha = \sum_{j\in J_0}A_jx_{\alpha-e_j} + \sum_{j\in J_+}A_jx_{\alpha-e_j}.
\end{equation}
If $j\in J_+$, then $\alpha-e_j\in\mathbb N^m$ and
\[ d(\alpha-e_j)=d(\alpha)-1. \]
Therefore the induction hypothesis gives
\[ x_{\alpha-e_j} = \sum_{i=1}^m \sum_{\gamma\in\mathcal R_i(\alpha-e_j)} W_\gamma A_i x_{\alpha-e_j-\gamma-e_i}. \]
Substituting this into the second sum in \eqref{prop31-split}, we obtain
\begin{align*}
	\sum_{j\in J_+}A_jx_{\alpha-e_j} &= \sum_{j\in J_+}\sum_{i=1}^m \sum_{\gamma\in\mathcal R_i(\alpha-e_j)} A_jW_\gamma A_i x_{\alpha-e_j-\gamma-e_i}.
\end{align*}
	
For each term set
\[ \beta=\gamma+e_j. \]
Then $\beta\in\mathcal R_i(\alpha)$ and $\beta_j>0$. Conversely, if $\beta\in\mathcal R_i(\alpha)$ and $\beta_j>0$, then $\alpha_j\geq2$ and
\[ \beta-e_j\in\mathcal R_i(\alpha-e_j). \]
Thus the change of variables above gives
\begin{align*}
	\sum_{j\in J_+}A_jx_{\alpha-e_j} &= \sum_{i=1}^m \sum_{\substack{\beta\in\mathcal R_i(\alpha)\\ \beta \neq 0}} \left( \sum_{\substack{1\leq j\leq m\\ \beta_j>0}}
	A_jW_{\beta-e_j} \right) A_i x_{\alpha-\beta-e_i}.
\end{align*}
By the word recursion \eqref{word-recursion},
\[ \sum_{\substack{1\leq j\leq m\\ \beta_j>0}} A_jW_{\beta-e_j} = W_\beta, \]
and hence
\begin{equation}\label{prop31-interior}
	\sum_{j\in J_+}A_jx_{\alpha-e_j} = \sum_{i=1}^m \sum_{\substack{\beta\in\mathcal R_i(\alpha)\\ \beta \neq 0}} W_\beta A_i x_{\alpha-\beta-e_i}.
\end{equation}
	
It remains to account for the first sum in \eqref{prop31-split}. Observe that
\[ 0 \in \mathcal R_i(\alpha) \quad\Longleftrightarrow\quad \alpha_i=1.\]
Consequently,
\begin{equation}\label{prop31-boundary}
	\sum_{j\in J_0}A_jx_{\alpha-e_j} = \sum_{i=1}^m \sum_{\substack{\beta\in\mathcal R_i(\alpha)\\ \beta=0}} W_\beta A_i x_{\alpha-\beta-e_i},
\end{equation}
since $W_0=I$.
	
Combining \eqref{prop31-interior} and \eqref{prop31-boundary} with \eqref{prop31-split} yields
\[ x_\alpha = \sum_{i=1}^m \sum_{\beta\in\mathcal R_i(\alpha)} W_\beta A_i x_{\alpha-\beta-e_i},\]
which is the desired formula.
\end{proof}

\medskip

To organize sums over $\mathbb{N}^m$ it is convenient to group multi-indices by their $\ell^1$-level sets (``shells''). The next proposition formalizes this shell decomposition and provides a useful way to rewrite (and estimate) multi-index sums by summing first over the shell radius and then within each shell.

For $p\in\mathbb N_0$ we define the \emph{$p$-th shell} $S_p\subset\mathbb N_0^m$ as follows. Let
\begin{equation*}
I_p:=\{\alpha\in\mathbb N^m:\ |\alpha|=p+m\}
\end{equation*}
be the interior part of the shell, and for each $k=1,\ldots,m$, let
\begin{equation*}
B_{p,k}:=\{\beta\in\mathbb N_0^m:\ \beta_k=0,\ \beta_j\ge 1\ (j\neq k),\ |\beta|=p+m\}.
\end{equation*}
Set
\begin{equation*}
B_p:=\bigcup_{k=1}^m B_{p,k},
\qquad
S_p:=I_p\cup B_p.
\end{equation*}

\begin{Prop}\label{ym}
Let $p\in\mathbb N_0$ and define the shell sum
	\begin{equation}\label{Sp}
		y_p:=\sum_{\alpha\in S_p}x_\alpha.
	\end{equation}
Then $y_p$ satisfies the one-dimensional recurrence
	\begin{equation}\label{shell-recursion}
		y_{p+1} = \left(\sum_{k=1}^m A_k\right)y_p+b_p,
	\end{equation}
where the forcing term $b_p$ depends only on the boundary data. More precisely,
	\begin{equation}\label{bp}
		b_p = \sum_{k=1}^m \left[ \sum_{\alpha\in B_{p+1,k}}x_\alpha - \sum_{\substack{j=1\\ j\neq k}}^m \sum_{\alpha\in B_{p,j}}A_kx_\alpha \right].
	\end{equation}
\end{Prop}

\begin{proof}
By definition,
\[ S_{p+1}=I_{p+1}\cup B_{p+1}, \]
and the union is disjoint. Hence
\begin{equation}\label{yp1-split}
	y_{p+1} = \sum_{\alpha\in I_{p+1}}x_\alpha + \sum_{\alpha\in B_{p+1}}x_\alpha.
\end{equation}
For every $\alpha\in I_{p+1}$, the recurrence relation gives
\[ x_\alpha = \sum_{k=1}^m A_kx_{\alpha-e_k}. \]
Therefore,
\begin{equation}\label{interior-shell-sum}
	\sum_{\alpha\in I_{p+1}}x_\alpha = \sum_{k=1}^m A_k \sum_{\alpha\in I_{p+1}}x_{\alpha-e_k}.
\end{equation}
	
Fix $k\in\{1,\ldots,m\}$. As $\alpha$ ranges over $I_{p+1}$, the multi-index
\[ \gamma:=\alpha-e_k \]
has $|\gamma|=p+m$, satisfies $\gamma_j\geq1$ for $j\neq k$, and satisfies $\gamma_k\geq0$. Consequently,
\[ \{\alpha-e_k:\alpha\in I_{p+1}\} = I_p\cup B_{p,k}. \]
Hence
\begin{align}
	\sum_{\alpha\in I_{p+1}}x_{\alpha-e_k} &= \sum_{\gamma\in I_p}x_\gamma + \sum_{\gamma\in B_{p,k}}x_\gamma \nonumber\\
	&= y_p - \sum_{\substack{j=1\\j\neq k}}^m \sum_{\gamma\in B_{p,j}}x_\gamma,\label{shifted-shell}
\end{align}
since
\[ S_p = I_p \cup \bigcup_{j=1}^m B_{p,j}, \]
and these sets are pairwise disjoint.
	
Substituting \eqref{shifted-shell} into \eqref{interior-shell-sum}, we obtain
\begin{align*}
	\sum_{\alpha\in I_{p+1}}x_\alpha &= \sum_{k=1}^m A_ky_p - \sum_{k=1}^m \sum_{\substack{j=1\\j\neq k}}^m \sum_{\gamma\in B_{p,j}}A_kx_\gamma.
\end{align*}
Using this identity in \eqref{yp1-split} gives
\begin{align*}
	y_{p+1} &= \left(\sum_{k=1}^m A_k\right)y_p +\sum_{\alpha\in B_{p+1}}x_\alpha - \sum_{k=1}^m \sum_{\substack{j=1\\j\neq k}}^m \sum_{\alpha\in B_{p,j}}A_kx_\alpha.
\end{align*}
Since
\[ B_{p+1} = \bigcup_{k=1}^m B_{p+1,k} \]
is a disjoint union,
\[ \sum_{\alpha\in B_{p+1}}x_\alpha = \sum_{k=1}^m \sum_{\alpha\in B_{p+1,k}}x_\alpha. \]
Therefore
\[ y_{p+1} = \left(\sum_{k=1}^m A_k\right)y_p + \sum_{k=1}^m \left[ \sum_{\alpha\in B_{p+1,k}}x_\alpha - \sum_{\substack{j=1\\j\neq k}}^m
	\sum_{\alpha\in B_{p,j}}A_kx_\alpha \right], \]
which is precisely \eqref{shell-recursion}--\eqref{bp}.
\end{proof}

When $m=2$, the shell is simply
\[ S_p=\{(i,j)\in\mathbb N_0^2:i+j=p+2\}, \]
so that $y_p$ is the sum of the states along the $(p+2)$-nd anti-diagonal.

\medskip

The following corollary gives a simple and practical condition ensuring instability of the system. We will invoke this criterion repeatedly in the sequel, since it provides an efficient way to verify instability in a variety of situations.

\begin{cor} \label{unstable}
If $\rho(P(z)) > 1$ for some $z \in \mathbb{T}^m$, then ${\cal S}(A_1, \dots, A_m)$ is unstable.
\end{cor}

\begin{proof}
By Remark~\ref{rotation-invariance}, after replacing $A_k$ by $z_kA_k$ if necessary, it is enough to consider the case
\[ \rho(A_1+\cdots+A_m)>1. \]
Since
\[ P(\mathbf 1)=A_1+\cdots+A_m, \]
the assumption becomes
\[ \rho(P(\mathbf 1))>1. \]
We argue by contradiction and assume that $\mathcal S(A_1,\ldots,A_m)$ is Lyapunov stable. By linearity and Lyapunov stability, there exists a constant $C>0$ such
that every solution corresponding to boundary data satisfying
\[ \|x\|_{\infty, \mathcal H} \leq 1 \]
satisfies
\begin{equation} \label{Lyap bound}
	\|x_\alpha\|\leq C \qquad\text{for every }\alpha\in\mathbb N^m.
\end{equation}
Indeed, apply the definition of Lyapunov stability with $\varepsilon=1$ and then use the homogeneity of the recurrence.
	
We first show that the word sums $W_\beta$ are uniformly bounded. Fix $\beta\in\mathbb N_0^m$ and $i\in\{1,\ldots,m\}$. Prescribe boundary data supported at the
single point $\mathbf 1-e_i$ by setting
\[ x_{\mathbf 1-e_i}=u,\qquad x_\gamma=0\quad (\gamma\in\mathcal H,\ \gamma\neq\mathbf 1-e_i), \]
where $\|u\|\leq1$. Applying Proposition~\ref{bigform} at $\alpha=\beta+\mathbf 1$, the only nonzero boundary contribution is the one corresponding to the $i$th face and the multi-index $\beta$. Hence
\[ x_{\beta+\mathbf 1}=W_\beta A_i u.\]
It follows from \eqref{Lyap bound} that
\begin{equation} \label{Lyap bound 2}
	\|W_\beta A_i\|\leq C \qquad \text{for every }\beta\in\mathbb N_0^m\text{ and }i=1,\ldots,m.
\end{equation}
	
For $\beta\neq0$, partitioning the words occurring in $W_\beta$ according to their final letter gives
\[ W_\beta = \sum_{\substack{1\leq i\leq m\\ \beta_i>0}} W_{\beta-e_i}A_i.\]
Therefore, by \eqref{Lyap bound 2},
\[ \|W_\beta\|\leq mC. \]
Since $W_0=I$, there exists a constant $C_0>0$ such that
\begin{equation} \label{Lyap bound 3}
	\|W_\beta\|\leq C_0 \qquad \text{for every }\beta\in\mathbb N_0^m.
\end{equation}
	
Now, for every $r \in \mathbb N$,
\[ P(\mathbf 1)^r = (A_1+\cdots+A_m)^r = \sum_{\substack{\beta\in\mathbb N_0^m\\|\beta|=r}}W_\beta.\]
Using \eqref{Lyap bound 3} and the fact that the number of multi-indices $\beta\in\mathbb N_0^m$ satisfying $|\beta|=r$ is $\binom{r+m-1}{m-1}$, we obtain
\begin{equation} \label{Lyap bound 4}
	\|P(\mathbf 1)^r\| \leq C_0\binom{r+m-1}{m-1}=O(r^{m-1}).
\end{equation}
	
On the other hand, since $\rho(P(\mathbf 1))>1$, the matrix $P(\mathbf 1)$ has an eigenvalue $\lambda$ with $|\lambda|>1$ and a corresponding nonzero eigenvector
$v$. Hence
\[ \|P(\mathbf 1)^r\| \geq \frac{\|P(\mathbf 1)^r v \|}{\|v\|} = |\lambda|^r,\]
which grows exponentially with $r$. This contradicts the polynomial bound \eqref{Lyap bound 4}. Thus $\mathcal S(A_1,\ldots,A_m)$ is not Lyapunov stable and is therefore unstable.
\end{proof}

\section{A generalized Fornasini--Marchesini theorem}

In this section we state and prove our main result, namely an extension of the classical Fornasini--Marchesini theorem to the general setting considered in this paper. Roughly speaking, the extension provides a corresponding stability/instability characterization for the multi-index (or $m$-dimensional) systems under our hypotheses, and recovers the usual Fornasini--Marchesini statement as a particular case. For $m=2$, under the bounded-boundary/Pringsheim framework used here, the corresponding result was established in~\cite{BoMc26}. Theorem~\ref{bigTH} extends this characterization to an arbitrary number of discrete directions.  Alternative approaches to the stability analysis of multidimensional discrete state-space systems have been developed using Lyapunov and linear matrix inequality methods; see, for example, Elosmani et al.~\cite{ElosmaniEtAl2021} and Kanellakis et al.~\cite{KanellakisEtAl2024}. These works address more general multidimensional state-space formulations, whereas our focus is the homogeneous Fornasini--Marchesini recurrence with uniformly bounded boundary data and Pringsheim convergence. To the best of our knowledge, this extension has not previously been established for the bounded-boundary/Pringsheim framework, and it will serve as the central tool for the applications and examples developed later on.

We now characterize asymptotic stability in terms of the matrix pencil $P(z)$ introduced in Section~2. The relevant sets are the distinguished boundary $\mathbb T^m$ and the closed polydisk $\overline{\mathbb D}^{\,m}$.

\begin{theo}[Generalized Fornasini--Marchesini theorem] \label{bigTH}
Let $A_1,\ldots,A_m\in\mathcal M_n(\mathbb C)$. The following conditions are equivalent:
\begin{enumerate}
	\item[(a)] The system $\mathcal S(A_1,\ldots,A_m)$ is asymptotically stable.
	\item[(b)] $\rho(P(z))<1 \qquad\text{for every }z\in\mathbb T^m$.
	\item[(c)] $\rho(P(z))<1 \qquad\text{for every }z\in\overline{\mathbb D}^{\,m}$.
	\item[(d)] $\det(I-P(z))\neq0 \qquad\text{for every }z\in\overline{\mathbb D}^{\,m}$.
\end{enumerate}
\end{theo}

\medskip

\begin{proof}
$[(a) \Rightarrow (b)]$.  If $\rho( P(z))>1$, then $x=0$ is unstable by Corollary \ref{unstable}.  If $\rho(P(z))=1$, then $P(z)$ has  an eigenvalue $\lambda$ with $|\lambda| = 1$ and corresponding eigenvector $v \neq 0$ such that 
\begin{equation*}
	P(z) v = \lambda v.
\end{equation*} 
Define new matrices $A_k'= \lambda^{-1}z_kA_k$ for $k=1,\ldots,m$. Then $A_1'+ \cdots+ A_m'$  has eigenvalue $1$ with eigenvector $v$. It follows that 
\[ x_{\alpha} = v  \qquad \forall \alpha \in \N^m \]
solves $\mathcal{S}(A_1',\ldots,A_m')$ with boundary conditions $x_{\beta}=v$ for all $\beta \in \mathcal H$. By linearity, any nonzero scalar multiple of this constant solution is also a solution. Thus, arbitrarily small bounded boundary data can produce solutions that do not converge to zero, so $x=0$ is not attractive, contradicting asymptotic stability. Therefore, $(a)$ implies $(b)$. 
	
\medskip
	
\noindent $[(b) \Rightarrow (c).]$ Assume that
\[ \rho(P(z))<1 \qquad\text{for every }z\in\mathbb T^m. \]
Consider the function
\[ u(z):=\log \rho(P(z)), \qquad z\in\mathbb C^m \]
with $\log 0 := -\infty$. 

Recall that a function on $\C^m$ is plurisubharmonic if its restriction to every complex affine line is subharmonic. By Vesentini's theorem (see \cite{ve68}, pp. 427--429), applied to the restriction of $P$ to each such line, $u$ is plurisubharmonic on $\mathbb C^m$. Hence, by the one-variable maximum principle for plurisubharmonic functions applied successively in each coordinate,
\[ \max_{z\in\overline{\mathbb D}^{\,m}}u(z) = \max_{z\in\mathbb T^m}u(z). \]
By assumption,
\[ \max_{z\in\mathbb T^m}u(z)<0.\]
Therefore, $u(z)<0$ for every $z\in\overline{\mathbb D}^{\,m}$,  which is equivalent to
\[ \rho(P(z))<1 \qquad \text{for every }z\in\overline{\mathbb D}^{\,m}. \]
This proves (c).
	
\medskip
	
\noindent $[(c) \Rightarrow (d).]$ Assume that $\rho(P(z))<1$ for every $z\in\overline{\mathbb D}^{\,m}$.  If
\[ \det(I-P(z_0))=0 \]
for some $z_0\in\overline{\mathbb D}^{\,m}$, then $1$ is an eigenvalue of $P(z_0)$. Hence
\[ \rho(P(z_0))\geq 1, \]
contradicting condition (c). Therefore
\[ \det(I-P(z))\neq0 \qquad \text{for every }z\in\overline{\mathbb D}^{\,m}. \]
This proves $(d)$.
	
\medskip
	
\noindent $[(d) \Rightarrow (a)]$. Assume that
\[ \det\!\bigl(I-P(z))\bigr)\neq 0 \qquad \text{for all } (z_1,\ldots,z_m)\in \overline{\mathbb{D}}^m . \]
Hence
\[ F(z):=\bigl(I-P(z) \bigr)^{-1} \]
is holomorphic on an open neighborhood of the closed polydisk $ \overline{\mathbb{D}}^m$. Since $\overline{\mathbb{D}}^m$ is compact, there exists $R>1$ such that $F$ is holomorphic on the open polydisk $\{z\in\mathbb{C}^m:\ |z_k|<R,\ k=1,\ldots,m\}$. Therefore $F$ admits a power-series expansion
\begin{equation} \label{Series-new}
	F(z)=\sum_{\alpha\in\mathbb{N}_0^m} C_\alpha z^\alpha
\end{equation}
which converges absolutely on $|z_k|<R$.
	
We claim that $C_\alpha=W_\alpha$ for every $\alpha\in\mathbb{N}_0^m$. Indeed, for $|z_k|$ sufficiently small we have the Neumann-series expansion
\[ F(z)=\sum_{k=0}^\infty P(z)^k. \]
Expanding each power $P(z)^k$, the coefficient of the monomial $z^\alpha=z_1^{\alpha_1}\cdots z_m^{\alpha_m}$ is precisely the sum of all words in the matrices $A_1,\ldots,A_m$ in which $A_k$ appears exactly $\alpha_k$ times, namely $W_\alpha$. Thus, in a neighborhood of the origin,
\[ F(z)=\sum_{\alpha\in\mathbb{N}_0^m} W_\alpha z^\alpha . \]
By uniqueness of the Taylor expansion of a holomorphic function, it follows that $C_\alpha = W_\alpha$ for all $\alpha$. Consequently,
\begin{equation}  \label{Series-W}
	F(z) = \sum_{\alpha\in\mathbb{N}_0^m} W_\alpha z^\alpha
\end{equation}
throughout the polydisk of convergence $|z_k|<R$.
	
Fix $r \in (1,R)$. Since the series in \eqref{Series-W} converges absolutely on the closed polytorus $\{|z_1|=\cdots=|z_m|=r\}$, we have
\begin{equation} \label{MR-new}
	M_r:=\sum_{\alpha\in\mathbb{N}_0^m} \|W_\alpha\|\,r^{|\alpha|}<\infty.
\end{equation}
In particular,
\begin{equation} \label{coef-decay}
	\|W_\alpha\|\le M_r \, r^{-|\alpha|} \qquad \forall \alpha\in\mathbb{N}_0^m.
\end{equation}
Fix $\alpha\in\mathbb{N}^m$. By Proposition~\ref{bigform},
\[ x_\alpha=\sum_{i=1}^m \sum_{\beta\in\mathcal R_i(\alpha)} W_\beta A_i x_{\alpha-\beta-e_i}.\]
For every $\beta\in\mathcal R_i(\alpha)$, the multi-index $\alpha-\beta-e_i$ belongs to $\mathcal H$, so
\[ \|x_{\alpha-\beta-e_i}\|\le \|x\|_{\infty,\mathcal H}. \]
Hence
\begin{equation} \label{basic-est}
	\|x_\alpha\| \le \|x\|_{\infty,\mathcal H} \sum_{i=1}^m \|A_i\|\sum_{\beta\in\mathcal R_i(\alpha)} \|W_\beta\|.
\end{equation}
	
To prove Lyapunov stability, note that \eqref{MR-new} implies
\[ \sum_{\alpha\in\mathbb{N}_0^m}\|W_\alpha\|<\infty \]
because $r>1$. Therefore, from \eqref{basic-est},
\[ \|x_\alpha\| \le \|x\|_{\infty,\mathcal H} \Bigl(\sum_{i=1}^m \|A_i\|\Bigr) \sum_{\gamma\in\mathbb{N}_0^m}\|W_\gamma\| =: C\,\|x\|_{\infty,\mathcal H}, \]
where $C<\infty$ is independent of $\alpha$. If $C=0$, the conclusion is immediate. Otherwise, given $\varepsilon>0$, choosing $\delta=\varepsilon/C$ shows that $\|x\|_{\infty,\mathcal H}<\delta$ implies
$\|x_\alpha\|<\varepsilon$ for all $\alpha\in\mathbb{N}^m$. Thus $x\equiv 0$ is Lyapunov stable.
	
It remains to prove attractivity in the sense of Pringsheim. Using \eqref{coef-decay}, for each fixed $i$ and each $\beta\in\mathcal R_i(\alpha)$ we have $\beta_i=\alpha_i-1$, hence 
\[ \|W_\beta\| \leq M_r \, r^{-|\beta|} = M_r \, r^{1-\alpha_i}\prod_{k\neq i} r^{-\beta_k}. \] 
Summing over $\beta\in\mathcal R_i(\alpha)$, we obtain
\[ \sum_{\beta\in\mathcal R_i(\alpha)} \|W_\beta\| \leq M_r \,r^{1-\alpha_i} \prod_{k\neq i}\Bigl(\sum_{j=0}^{\alpha_k-1} r^{-j}\Bigr)
	\leq M_r \,\Bigl(\frac{r}{r-1}\Bigr)^{m-1} r^{1-\alpha_i}. \]
Substituting this into \eqref{basic-est} yields
\begin{equation} \label{Pringsheim-est}
	\|x_\alpha\| \leq M_r \Bigl(\frac{r}{r-1}\Bigr)^{m-1} \|x\|_{\infty,\mathcal H} \sum_{i=1}^m \|A_i\|\,r^{\,1-\alpha_i}.
\end{equation}
Now let $\alpha\to\infty$ in the sense of Pringsheim, i.e. $\alpha_i\to\infty$ for every $i=1,\ldots,m$. Since $r>1$, each term $r^{1-\alpha_i}\to 0$, and therefore the right-hand side of \eqref{Pringsheim-est} tends to $0$. It follows that $x_\alpha\to 0$ in the sense of Pringsheim. We have shown both Lyapunov stability and attractivity, so $\mathcal S(A_1,\ldots,A_m)$ is asymptotically stable. This proves $(d)\Rightarrow(a)$.
\end{proof}

We now derive two simple consequences of Theorem~\ref{bigTH} that yield convenient sufficient conditions for the asymptotic stability of the system $\mathcal{S}(A_1,\ldots,A_m)$. The first provides a criterion in terms of the norms of $A_1,\ldots,A_m$, valid for any induced matrix norm.

\begin{cor}
If $\|A_{1}\|+\cdots+\|A_{m}\|<1$, then $\mathcal S(A_{1},\ldots,A_{m})$ is asymptotically stable. Moreover, $\mathcal S(A_{1},\ldots,A_{m})$ is Lyapunov stable whenever $\|A_{1}\|+\cdots+\|A_{m}\|\le 1$.
\end{cor}

\begin{proof}
Assume first that
\[ \|A_1\|+\cdots+\|A_m\|<1. \]
For every $z=(z_1,\ldots,z_m)\in\mathbb T^m$, we have
\[ \rho(P(z)) \leq \|P(z)\| = \left\|\sum_{k=1}^m z_kA_k\right\| \leq \sum_{k=1}^m \|A_k\| <1. \]
Therefore the spectral-radius condition in Theorem~\ref{bigTH} is satisfied, and the system $\mathcal S(A_1,\ldots,A_m)$ is asymptotically stable.
	
Now suppose that
\[ \|A_1\|+\cdots+\|A_m\| \leq 1.\]
Let $(x_\beta)_{\beta\in\mathcal H}$ be bounded boundary data, and set
\[ M:=  \|x\|_{\infty, \mathcal H}. \]
We claim that
\[ \|x_\alpha\|\leq M \qquad \text{for every }\alpha\in\mathbb N^m.\]
We argue by strong induction on $|\alpha|$. Let $\alpha\in\mathbb N^m$, and assume that
\[ \|x_\gamma\|\leq M \]
for every $\gamma\in\mathbb N^m$ with $|\gamma|<|\alpha|$. For each $k\in\{1,\ldots,m\}$, the predecessor $\alpha-e_k$ either belongs to $\mathcal H$ or belongs to
$\mathbb N^m$ and satisfies
\[ |\alpha-e_k|=|\alpha|-1. \]
Hence, either by the definition of $M$ or by the induction hypothesis,
\[ \|x_{\alpha-e_k}\|\leq M \qquad (k=1,\ldots,m). \]
Using the recurrence relation, we therefore obtain
\[ \|x_\alpha\| \leq \sum_{k=1}^m \|A_k\|\,\|x_{\alpha-e_k}\| \leq \left(\sum_{k=1}^m\|A_k\|\right)M \leq M. \]
Thus
\[ \sup_{\alpha\in\mathbb N^m}\|x_\alpha\| \leq \|x\|_{\infty, \mathcal H}. \]
In particular, given $\varepsilon>0$, choosing $\delta=\varepsilon$ in Definition~\ref{def:stability} proves that the system  $\mathcal S(A_1,\ldots,A_m)$ is Lyapunov stable.
\end{proof}

\medskip

The second result is formulated in terms of the entrywise absolute value of a matrix $A=[a_{ij}]\in \mathcal{M}_n(\C)$. We denote by $|A|$ the matrix whose $(i,j)$-entry is $|a_{ij}|$.

\begin{cor}
If $\rho(|A_1|+ \cdots + |A_m| ) <1$, then $\mathcal{S}(A_1,\ldots,A_m)$ is asymptotically stable. 
\end{cor}

\begin{proof}
For every $z \in \mathbb{T}^m$, we have by Theorem 8.1.18 in \cite{HoJo85} that 
\begin{align*}
	\rho(P(z)) & \leq \rho(  |P(z)|)   \leq  \rho( |A_1|+ \cdots + |A_m| )< 1. 
\end{align*}
It then follows by Theorem \ref{bigTH} that $\mathcal{S}(A_1,\ldots,A_m)$ is asymptotically stable. 
\end{proof}

The next corollary implies that asymptotic stability is maintained if the matrices $A_k$ are slightly perturbed. That is, if ${\cal S}(A_1, \ldots,A_m)$ is asymptotically stable, then ${\cal S}(A_1',\ldots,A_m')$ will be asymptotically stable, provided that $\|A_k-A_k'\|$ is small enough for every $k=1,\ldots,m$.

\begin{cor} \label{open}
Let 
\[ \mathcal{A} := \{ (A_1,\ldots,A_m) : \mathcal{S}(A_1, \ldots,A_m) \text{ is asymptotically stable} \} \]
Then $\mathcal A$ is  open, star-shaped, and unbounded for $n\geq 2$.
\end{cor}
\begin{proof}
By Theorem~\ref{bigTH}, 
\[ \mathcal{A} := \{ (A_1,\ldots,A_m) : \rho(P(z)) < 1 \text{ for all } z \in \mathbb{T}^m \}.\]
Since the spectral radius is a continuous function of matrices and the torus $\mathbb{T}^m$ is compact, the function
\[ f(A_1,\ldots,A_m) = \max_{z \in \mathbb{T}^m} \rho(P(z)) \] 
is continuous, and hence $\mathcal{A} = f^{-1}((-\infty,1))$ is an open set.
	
Recall $\mathcal A$ is said to be star-shaped with respect to the origin if $(tA_1,\ldots,tA_m) \in \mathcal A$ for every $0 \leq t \leq 1$ whenever $(A_1,\ldots,A_m) \in  \mathcal A$. The homogeneity of the spectral radius implies that 
\[  \rho( z_1 (tA_1) + \cdots+ z_m (tA_m)) = t \rho(z_1 A_1  + \cdots + z_m A_m) <1 \]
for every $t \in [0,1]$ and every $(z_1,\ldots,z_m) \in \mathbb T^m$. Hence,  $(tA_1,\ldots,tA_m) \in \mathcal{A}$ whenever $(A_1,\ldots,A_m) \in \mathcal{A}$, i.e. $\mathcal{A}$ is star-shaped with center at the origin. 
	
Finally, if $n\geq 2$, let $N \in \mathcal{M}_n(\C)$ be a  single Jordan block with zeros along the diagonal. For any $t>0$, set $A_1=tN$ and $A_k=0$ for $k=2,\ldots,m$. It follows that 
\[  \rho(P(z)) = \rho( tN) =0 <1, \]
for every $z \in \mathbb{T}^m$, so $(A_1, \ldots,A_m) \in \mathcal{A}$.  But $\| A_1\|  = t\|N\| \rightarrow \infty \ (t\rightarrow\infty)$, so $\mathcal{A}$ is unbounded.
\end{proof}

\section{The critical spectral boundary}

The preceding results completely characterize the stability of $\mathcal S(A_1,\ldots,A_m)$ away from the critical spectral regime. Indeed, if
\[ \max_{z\in\mathbb T^m}\rho(P(z))<1, \]
then the system is asymptotically stable by Theorem~\ref{bigTH}, whereas if
\[ \rho(P(z))>1 \]
for some $z\in\mathbb T^m$, then the system is unstable by Corollary~\ref{unstable}. It therefore remains to examine the borderline case
\[ \max_{z\in\mathbb T^m}\rho(P(z))=1. \]
In this section we study this critical regime and show that the Jordan structure of the unimodular eigenvalues of the pencil $P(z)$ provides an important additional obstruction to Lyapunov stability.

\begin{theo}\label{defective-pencil}
Let $A_1,\ldots,A_m\in\mathcal M_n(\C)$, with $m\geq2$. Suppose that for some $z \in\mathbb T^m$, the pencil $P(z)$ has a nonsemisimple eigenvalue $\lambda$ with $|\lambda|=1$. Then $\mathcal S(A_1,\ldots,A_m)$ is unstable.
\end{theo}

\begin{proof}
If $\rho(P(z))>1$, the conclusion follows immediately from Corollary~\ref{unstable}. We may therefore assume that
\[ \rho(P(z))=1. \]
	
By Remark~\ref{rotation-invariance}, unimodular rescaling of the coefficient matrices does not change the stability type of the system. Replacing $A_k$ by
\[ B_k:=\lambda^{-1}z_kA_k, \qquad k=1,\ldots,m, \]
we may therefore assume that
\[ S:=B_1+\cdots+B_m \]
has the nonsemisimple eigenvalue $1$. Thus there exist linearly independent vectors $v,w\in\C^n$ such that
\begin{equation}\label{Jordan-chain-S}
	Sv=v, \qquad Sw=w+v.
\end{equation}
	
Suppose, toward a contradiction, that $\mathcal S(B_1,\ldots,B_m)$ is Lyapunov stable. Since the system is linear, every bounded boundary datum then gives rise to a bounded solution. Indeed, this follows by scaling the boundary data in the definition of Lyapunov stability. Prescribe the constant boundary data
\[ x_\beta=w, \qquad \beta\in\mathcal H. \]
Let
\[ y_p=\sum_{\alpha\in S_p}x_\alpha \]
be the shell sums introduced in Proposition~\ref{ym}. By the assumed Lyapunov stability, there exists $C>0$ such that
\[ \|x_\alpha\|\leq C \qquad\text{for all }\alpha\in\mathbb N^m. \]
Enlarging $C$, if necessary so that $C \geq \|w\|$, we obtain
\[ \| x_\alpha\| \leq C \qquad (\alpha \in S_p, p \geq 0).\]
Hence
\begin{equation}\label{shell-upper}
	\|y_p\| \leq C|S_p| =O(p^{m-1}),
\end{equation}
since
\[ |I_p| = \binom{p+m-1}{m-1} \]
and
\[ |B_p| = m\binom{p+m-1}{m-2}. \]
	
We now use the recurrence of Proposition~\ref{ym}. Set
\[ h_p:=|B_{p,k}| = \binom{p+m-1}{m-2},\]
which is independent of $k$. Since the boundary data are equal to $w$, the forcing term in the shell recursion is
\[ \begin{aligned}
         b_p &= \sum_{k=1}^m \left[ \sum_{\alpha\in B_{p+1,k}}w - \sum_{j \neq k} \sum_{\alpha \in B_{p,j}}  B_kw \right] \\
		&= mh_{p+1}w-(m-1)h_pSw.
\end{aligned} \]
Using \eqref{Jordan-chain-S}, this becomes
\begin{equation}\label{bp-jordan}
	b_p = d_pw-(m-1)h_pv, \qquad d_p:=mh_{p+1}-(m-1)h_p.
\end{equation}
Notice that
\begin{equation} \label{hp-jordan} 
	h_p\sim \frac{p^{m-2}}{(m-2)!}, \qquad d_p\sim \frac{p^{m-2}}{(m-2)!}.
\end{equation}
	
The initial shell sum $y_0$ belongs to $\operatorname{span}\{v,w\}$. Indeed, the boundary values in $S_0$ are multiples of $w$, while the unique interior value is
\[ x_{\mathbf 1}=Sw=w+v. \]
Since both $S$ and the forcing terms $b_p$ leave $\operatorname{span}\{v,w\}$ invariant, we may write
\[ y_p=a_pw+c_pv. \]
The shell recursion
\[ y_{p+1}=Sy_p+b_p\]
and \eqref{Jordan-chain-S}--\eqref{bp-jordan} give
\begin{align}
	a_{p+1} &=a_p+d_p,     \label{ap-rec}\\
	c_{p+1} &=c_p+a_p-(m-1)h_p.    \label{cp-rec}
\end{align}
	
From \eqref{ap-rec} and the asymptotic behavior of $d_p$,
\[ a_p = a_0+\sum_{r=0}^{p-1}d_r \sim \frac{p^{m-1}}{(m-1)!}. \]
On the other hand,
\[ h_p=O(p^{m-2}), \]
so \eqref{cp-rec} yields
\[ c_{p+1}-c_p = a_p-(m-1)h_p \sim \frac{p^{m-1}}{(m-1)!}. \]
Summing once more,
\begin{equation}\label{cp-growth}
	c_p \sim \frac{p^m}{m!}.
\end{equation}
	
Since $v$ and $w$ are linearly independent, there exists a linear functional $\varphi\in(\C^n)^*$ such that
\[ \varphi(v)=1, \qquad \varphi(w)=0. \]
Consequently,
\[ |\varphi(y_p)|=|c_p|, \]
and therefore, by \eqref{cp-growth},
\[ \|y_p\| \geq \frac{|c_p|}{\|\varphi\|} \geq C_1p^m\]
for all sufficiently large $p$ and some $C_1>0$. This contradicts the estimate \eqref{shell-upper}, which gives
\[ \|y_p\|=O(p^{m-1}). \]
	
Hence $\mathcal S(B_1,\ldots,B_m)$ cannot be Lyapunov stable. By Remark~\ref{rotation-invariance}, the original system $\mathcal S(A_1,\ldots,A_m)$ is also unstable.
\end{proof}

Combining Theorem~\ref{defective-pencil} with the instability criterion of Corollary~\ref{unstable}, we obtain a convenient necessary spectral condition for Lyapunov stability. It summarizes the two obstructions that can occur on the torus: spectral radius larger than one, and defective unimodular eigenvalues in the critical case.

\begin{cor}\label{necessary-lyap}
If $\mathcal S(A_1,\ldots,A_m)$ is Lyapunov stable, then, for every $z \in \mathbb T^m$,
\[ \rho(P(z))\leq 1. \]
Moreover, every eigenvalue
\[ \lambda\in \sigma(P(z))\cap\mathbb T \]
is semisimple.
\end{cor}

\begin{proof}
If
\[ \rho(P(z))>1 \]
for some $z\in\mathbb T^m$, then
$\mathcal S(A_1,\ldots,A_m)$ is unstable by
Corollary~\ref{unstable}. Hence Lyapunov stability implies
\[ \rho(P(z))\leq1 \qquad \text{for every }z\in\mathbb T^m. \]
	
If, for some $z\in\mathbb T^m$, the matrix $z_1A_1+\cdots+z_mA_m$ had a nonsemisimple eigenvalue $\lambda\in\mathbb T$, then Theorem~\ref{defective-pencil} would imply that the system $\mathcal S(A_1,\ldots,A_m)$ is unstable. Therefore every unimodular eigenvalue of every toral pencil must be semisimple.
\end{proof}

\begin{ex}\label{ex1}
The critical spectral condition does not by itself imply instability. Indeed, let
\[ A_1:=I_n, \qquad A_2=\cdots=A_m:=0. \]
Then
\[ P(z)=z_1I_n,\]
and hence	
\[ \rho(P(z))=|z_1|=1 \qquad \text{for every }z\in\mathbb T^m.\]
Thus the system lies on the critical spectral boundary.
	
The recurrence reduces to
\[ x_\alpha=x_{\alpha-e_1}, \qquad \alpha\in\mathbb N^m. \]
Therefore every interior value is obtained by propagating a boundary value in the first coordinate direction, and
\[ \sup_{\alpha\in\mathbb N^m}\|x_\alpha\| \leq \|x\|_{\infty, \mathcal H}. \]
Hence the system is Lyapunov stable.
	
On the other hand, the system is not attractive. For example, prescribe the constant boundary data
\[ x_\beta=v, \qquad \beta\in\mathcal H, \]
for some nonzero $v\in\mathbb C^n$. Then
\[ x_\alpha=v \qquad \text{for every }\alpha\in\mathbb N^m, \]
so $x_\alpha$ does not converge to zero in the sense of Pringsheim. Since $v$ may be chosen with arbitrarily small nonzero norm, the system is not attractive. Consequently,  $\mathcal S(A_1,\ldots,A_m)$	is neutrally stable.
\end{ex}

\begin{ex}
The converse of Corollary~\ref{necessary-lyap} does not hold, even for commuting coefficient matrices. Indeed, consider the two-dimensional system $\mathcal S(A_1,A_2)$ with
\[ A_1=\frac12 I+N, \qquad A_2=\frac12 I-N, \qquad N=
	\begin{bmatrix}
		0&1\\
		0&0
	\end{bmatrix}. \]
The matrices $A_1$ and $A_2$ commute, and for $(z_1,z_2)\in\mathbb T^2$ the associated pencil is
\[ P(z_1,z_2) = \frac{z_1+z_2}{2}\,I+(z_1-z_2)N. \]
Since $N$ is nilpotent, the only eigenvalue of $P(z_1,z_2)$ is
\[ \lambda(z_1,z_2)=\frac{z_1+z_2}{2}, \]
with algebraic multiplicity two. Consequently,
\[ \rho(P(z_1,z_2)) = \left|\frac{z_1+z_2}{2}\right| \leq1 \qquad \text{for every }(z_1,z_2)\in\mathbb T^2. \]
Moreover, equality can occur only when $z_1=z_2$. Indeed, $|z_1+z_2|=2$ for $z_1,z_2\in\mathbb T$ if and only if $z_1=z_2$. In that case,
\[ P(z_1,z_1)=z_1I, \]
so its unimodular eigenvalue is semisimple. Thus this example satisfies both necessary conditions in Corollary~\ref{necessary-lyap}: every toral pencil has spectral radius at most one, and every unimodular eigenvalue of every toral pencil is semisimple.
	
Nevertheless, as shown in our earlier two-dimensional study \cite{BoMc26}, the system $\mathcal S(A_1,A_2)$ is unstable. Hence the pointwise spectral conditions
\[ \rho(P(z)) \leq 1 \quad \text{and} \quad \text{every } \lambda \in  \sigma(P(z))\cap\mathbb T \text{ is semisimple} \]
for every $z\in\mathbb T^m$ are necessary, but not sufficient, for Lyapunov stability. In particular, the critical stability problem depends on more than the Jordan structure of each individual toral pencil.
\end{ex}

The next example illustrates the defective critical situation described in Theorem~\ref{defective-pencil}. In fact, the toral spectral radius is identically equal to one, yet the system is unstable because a critical pencil has a nonsemisimple unimodular eigenvalue.

\begin{ex}\label{Ex 2}
Let $m\geq2$ and $n=m$. Define
\[ A_1:=I_m, \qquad A_k:=J,\quad k=2,\ldots,m, \]
where $J\in\mathcal M_m(\mathbb C)$ is the nilpotent Jordan block with ones on the superdiagonal. Then, for every $z=(z_1,\ldots,z_m)\in\mathbb T^m$,
\[ P(z) = z_1I_m+\left(\sum_{k=2}^m z_k\right)J. \]
Since $J$ is strictly upper triangular, $P(z)$ is upper triangular with all diagonal entries equal to $z_1$. Hence
\[ \rho(P(z))=|z_1|=1 \qquad \text{for every }z\in\mathbb T^m. \]
Thus the system lies entirely on the critical spectral boundary. In particular, at $z=(1,\ldots,1)$,
\[ P(1,\ldots,1)=I_m+(m-1)J. \]
The eigenvalue $1$ is nonsemisimple, since $J\neq0$. Therefore, by Theorem~\ref{defective-pencil},
\[ \mathcal S(A_1,\ldots,A_m) \]
is unstable.
\end{ex}

Theorem \ref{defective-pencil} shows that semisimplicity of the unimodular spectrum of the critical pencils is necessary for Lyapunov stability. In general this condition need not be sufficient. For simultaneously diagonalizable matrices, however, the joint spectral structure allows a complete description of the critical regime, which is the subject of the next section.

\section{The simultaneously diagonalizable  case}

We next consider the case in which the matrices $A_1,\ldots,A_m$ are simultaneously diagonalizable. Thus there exists an invertible matrix $X$ such that
\begin{equation}\label{diag}
	A_k=X\Lambda_kX^{-1}, \qquad k=1,\ldots,m,
\end{equation}
where
\[ \Lambda_k = \operatorname{diag}(a_{k,1},\ldots,a_{k,n}). \]
Here $a_{k,j}$ denotes the eigenvalue of $A_k$ corresponding to the $j$-th common eigenvector of $A_1,\ldots,A_m$. Consequently,
\[ P(z) =X\left(\sum_{k=1}^m z_k\Lambda_k\right)X^{-1}, \]
and therefore
\[ \rho(P(z)) = \max_{1\leq j\leq n} \left| \sum_{k=1}^m z_k a_{k,j} \right|. \]
Taking the maximum over $z\in\mathbb T^m$ gives
\[ \max_{z\in\mathbb T^m}\rho(P(z)) = \max_{1\leq j\leq n} \sum_{k=1}^m |a_{k,j}|. \]
Thus, in the simultaneously diagonalizable case, the critical spectral condition reduces to a collection of scalar conditions along the common eigendirections.

In Theorem~\ref{simultaneously-diagonalizable}, we obtain a complete characterization of the stability of $\mathcal S(A_1,\ldots,A_m)$ under these assumptions. We begin by analyzing the scalar case, in which the matrices $A_1,\ldots,A_m$ are replaced by complex numbers $a_1,\ldots,a_m$. Scalar Fornasini--Marchesini models and their asymptotic stability have also been studied in \cite{Ba19}. We record below the scalar classification needed for the bounded-boundary/Pringsheim framework considered here.

\begin{lem}  \label{scalar}
The system $\mathcal{S}(a_1, \ldots, a_m)$ is:
	\begin{itemize}
		\item[(a)] asymptotically stable if $|a_1| + \cdots + |a_m| <1$.
		\item[(b)] neutrally stable if $|a_1| + \cdots + |a_m| = 1$.
		\item[(c)] unstable if $|a_1| + \cdots + |a_m| >1.$
	\end{itemize}	
\end{lem}

\begin{proof}
For each $k$ with $a_k\neq0$, choose $z_k\in\mathbb T$ such that $z_ka_k=|a_k|$, and choose $z_k\in\mathbb T$ arbitrarily if $a_k=0$. Then
\[ z_1a_1+\cdots+z_ma_m = |a_1|+\cdots+|a_m|. \]
Together with the triangle inequality, this gives
\begin{equation}\label{max torus}
	\max_{z\in\mathbb T^m} \left| \sum_{k=1}^m z_ka_k \right| = \sum_{k=1}^m |a_k|.
\end{equation}
	
\noindent \textit{(a)} If
\[ \sum_{k=1}^m|a_k|<1, \]
then \eqref{max torus} implies
\[ \left|\sum_{k=1}^m z_ka_k\right|<1 \qquad \text{for every }z\in\mathbb T^m. \]
Therefore, by Theorem~\ref{bigTH}, $\mathcal S(a_1,\ldots,a_m)$ is asymptotically stable.
	
\noindent \textit{(b)} Assume that
\[ \sum_{k=1}^m|a_k|=1. \]
Let
\[ M:=\|x\|_{\infty,\mathcal H}. \]
We claim that
\[ |x_\alpha|\leq M \qquad \text{for every }\alpha\in\mathbb N^m. \]
Indeed, this follows by strong induction on $|\alpha|$. For each $k$, the predecessor $\alpha-e_k$ either belongs to $\mathcal H$ or is an interior multi-index of smaller order, and hence
\[ |x_{\alpha-e_k}|\leq M. \]
Consequently,
\[ |x_\alpha| \leq \sum_{k=1}^m|a_k|\,|x_{\alpha-e_k}| \leq \left(\sum_{k=1}^m|a_k|\right)M = M. \]
Thus the system is Lyapunov stable.
	
It remains to show that it is not attractive. Choose $z=(z_1,\ldots,z_m)\in\mathbb T^m$ as above, so that $z_ka_k=|a_k|$. By the unimodular invariance established in
Remark~\ref{rotation-invariance}, $\mathcal S(a_1,\ldots,a_m)$ has the same stability type as	
\[ \mathcal S(|a_1|,\ldots,|a_m|). \]
For the latter system, the constant boundary data
\[ x_\beta=1, \qquad \beta\in\mathcal H, \]
give the constant solution
\[ x_\alpha=1, \qquad \alpha\in\mathbb N^m,\]
because $\sum_{k=1}^m|a_k|=1$. By linearity, any nonzero scalar multiple of this constant solution is again a constant solution. Hence arbitrarily small bounded boundary data can produce solutions that do not converge to zero, and the system is not attractive. Therefore $\mathcal S(a_1,\ldots,a_m)$ is neutrally stable.
	
\noindent \textit{(c)} If
\[ \sum_{k=1}^m|a_k|>1, \]
then \eqref{max torus} yields a point $z\in\mathbb T^m$ such that
\[ \left|\sum_{k=1}^m z_ka_k\right|>1.\]
Corollary~\ref{unstable} therefore implies that $\mathcal S(a_1,\ldots,a_m)$ is unstable.
\end{proof}

\medskip

The simultaneous diagonalization in \eqref{diag} reduces the multidimensional system to a collection of independent scalar systems, one along each common eigendirection. Applying Lemma~\ref{scalar} to these scalar components yields the following complete classification.

\begin{theo} \label{simultaneously-diagonalizable}
Let $A_1,\ldots,A_m \in \mathcal{M}_n(\C)$ be simultaneously diagonalizable as in \eqref{diag}. The system $\mathcal{S}(A_1,\ldots,A_m)$ is:
\begin{itemize}
	\item[(a)] asymptotically stable if $\sum_{k=1}^m |a_{k,j}|<1$ for every $j=1,\ldots,n$. 
	\item[(b)] neutrally stable if $\sum_{k=1}^m |a_{k,j}| \leq 1$ for all $j$ and the equality holds for at least one  $j \in \{\ 1,\ldots,n\}$.
	\item[(c)] unstable if $\sum_{k=1}^m |a_{k,j}|>1$ for some $j \in \{ 1,\ldots,n\}$.
\end{itemize}
\end{theo}

\begin{proof}
Set
\[ x_\alpha':=X^{-1}x_\alpha. \]
Since $X$ is invertible, the transformation $x_\alpha\mapsto x_\alpha'$ preserves boundedness and Pringsheim convergence. Hence
\[ \mathcal S(A_1,\ldots,A_m) \]
has the same stability type as the transformed system.
	
Using $A_k=X\Lambda_kX^{-1}$, the recurrence \eqref{rec rel} becomes
\begin{equation}\label{rec rel'}
	x_\alpha' = \sum_{k=1}^m \Lambda_k x_{\alpha-e_k}'.
\end{equation}
Because each $\Lambda_k$ is diagonal, the system decouples coordinatewise. More precisely, for each $j=1,\ldots,n$, the $j$-th component $(x_\alpha')^{(j)}$ satisfies
\[ (x_\alpha')^{(j)} = \sum_{k=1}^m a_{k,j}(x_{\alpha-e_k}')^{(j)}, \]
that is, the scalar system
\[ \mathcal S(a_{1,j},\ldots,a_{m,j}). \]
Therefore the stability of the original system is determined by the stability of these $n$ scalar systems, and assertions (a)--(c) follow directly from Lemma~\ref{scalar}.
\end{proof}

\medskip

As a consequence, we obtain the following result, which gives a simple sufficient condition for the asymptotic and Lyapunov stability of $\mathcal{S}(A_1,\ldots,A_m)$ in terms of the sum of the spectral radii of the matrices $A_1,\ldots,A_m$.

\begin{cor}\label{sp radius comm}
Let $A_1,\ldots,A_m\in\mathcal M_n(\C)$ be simultaneously diagonalizable. Then:
\begin{enumerate}
	\item[(a)] If $\rho(A_1)+\cdots+\rho(A_m)<1$,  then $\mathcal S(A_1,\ldots,A_m)$ is asymptotically stable.
	\item[(b)] If $\rho(A_1)+\cdots+\rho(A_m)\leq1$,  then $\mathcal S(A_1,\ldots,A_m)$ is Lyapunov stable.
\end{enumerate}
\end{cor}

\begin{proof}
Since the matrices are simultaneously diagonalizable, there exists an invertible matrix $X$ such that
\[ A_k = X\Lambda_kX^{-1}, \qquad \Lambda_k = \operatorname{diag}(a_{k,1},\ldots,a_{k,n}), \]
for $k=1,\ldots,m$. For every $j=1,\ldots,n$ and $k=1,\ldots,m$, we have
\[ |a_{k,j}|\leq\rho(A_k).\]
Hence, for every common eigendirection $j$,
\begin{equation}\label{rho-sum-bound}
	\sum_{k=1}^m |a_{k,j}| \leq \sum_{k=1}^m\rho(A_k).
\end{equation}
	
If
\[ \sum_{k=1}^m\rho(A_k)<1, \]
then \eqref{rho-sum-bound} gives
\[ \sum_{k=1}^m|a_{k,j}|<1 \qquad \text{for every }j=1,\ldots,n.\]
Therefore, by Theorem~\ref{simultaneously-diagonalizable}(a), $\mathcal S(A_1,\ldots,A_m)$ is asymptotically stable. Similarly, if
\[ \sum_{k=1}^m\rho(A_k)\leq1, \]
then
\[ \sum_{k=1}^m|a_{k,j}|\leq1 \qquad \text{for every }j=1,\ldots,n. \]
By Theorem~\ref{simultaneously-diagonalizable}, each scalar system associated with a common eigendirection is either asymptotically stable or neutrally stable. In either case it is Lyapunov stable. Consequently, $\mathcal S(A_1,\ldots,A_m)$ is Lyapunov stable.
\end{proof}

It is natural to ask whether the sufficient condition in Corollary~\ref{sp radius comm} remains valid without the assumption of simultaneous diagonalizability. More precisely, in the critical regime
\[ \max_{z\in\mathbb T^m}\rho(P(z))=1, \]
does the condition
\[ \sum_{k=1}^m \rho(A_k)<1 \]
guarantee Lyapunov stability? The following example shows that the answer is negative.

\begin{ex} \label{ex 2}
For $m=2$, let
\[ J=
	\begin{bmatrix}
		1&1\\
		0&1
	\end{bmatrix},
	\qquad
	A_1=
	\begin{bmatrix}
		0&I_2\\
		0&0
	\end{bmatrix},
	\qquad
	A_2=
	\begin{bmatrix}
		0&0\\
		J&0
	\end{bmatrix}.	\]
Since $A_1^2=A_2^2=0$, both matrices are nilpotent and therefore
\[ \rho(A_1)+\rho(A_2)=0<1. \]
On the other hand,
\[ P(z_1,z_2) = z_1A_1+z_2A_2 =
	\begin{bmatrix}
		0&z_1I_2\\
		z_2J&0
	\end{bmatrix},\]
and hence
\[ P(z_1,z_2)^2 = z_1z_2
	\begin{bmatrix}
		J&0\\
		0&J
	\end{bmatrix}.\]
Since $\rho(J)=1$ and 
\[ \rho(P(z))^2 = \rho(P(z)^2) = |z_1z_2| \rho(J),\]
it follows that
\[ \rho(P(z_1,z_2))=1 \qquad \text{for every }(z_1,z_2)\in\mathbb T^2. \]
Thus the system lies in the critical regime. At $z_1=z_2=1$,
\[ P(1,1)=
	\begin{bmatrix}
		0&I_2\\
		J&0
	\end{bmatrix}.\]
Its characteristic polynomial is
\[ \det(\lambda^2 I_2-J) = (\lambda^2-1)^2, \]
so $1$ has algebraic multiplicity $2$. Moreover,
\[ P(1,1)
	\begin{bmatrix}
		x\\y
	\end{bmatrix}
	=
	\begin{bmatrix}
		x\\y
	\end{bmatrix}\]
implies $y=x$ and $Jx=x$. Since
\[ \dim\ker(J-I_2)=1, \]
the eigenvalue $1$ is defective. Therefore, by Theorem~\ref{defective-pencil}, the system $\mathcal S(A_1,A_2)$ is unstable.
\end{ex}

\section{Conclusion}

In this paper we have studied the stability of the multidimensional discrete system
\[ x_\alpha=\sum_{k=1}^m A_kx_{\alpha-e_k}, \qquad \alpha\in\mathbb N^m, \]
with uniformly bounded boundary data prescribed on the codimension-one coordinate faces and attractivity understood in the sense of Pringsheim convergence.

Our first contribution was to develop explicit tools for analyzing the multi-index recurrence. We introduced the word sums $W_\beta$, which encode all ordered products of the coefficient matrices with prescribed multiplicities, and used them to obtain a representation formula for the solution in terms of the boundary data. We also introduced a shell decomposition of the lattice and showed that the corresponding shell sums satisfy a one-dimensional affine recursion driven by $A_1+\cdots+A_m$.

The main result of the paper is a generalized Fornasini--Marchesini theorem for this $m$-dimensional setting. Writing
\[ P(z)=\sum_{k=1}^m z_kA_k, \]
we proved that asymptotic stability is equivalent to
\[ \rho(P(z))<1 \qquad \text{for every }z\in\mathbb T^m, \]
and also equivalent to the condition
\[ \det(I-P(z))\neq0 \qquad \text{for every }z\in\overline{\mathbb D}^{\,m}. \]
Thus the asymptotically stable region is completely characterized by the multivariate pencil $P(z)$.

We also investigated the critical spectral boundary
\[ \max_{z\in\mathbb T^m}\rho(P(z))=1. \]
In this regime, the spectral radius alone is no longer sufficient to determine Lyapunov stability. We showed that if a toral pencil $P(z)$ has a nonsemisimple unimodular eigenvalue, then the corresponding system is unstable. Consequently, Lyapunov stability requires both
\[ \rho(P(z))\leq1 \qquad (z\in\mathbb T^m) \]
and semisimplicity of every unimodular eigenvalue of every toral pencil. These conditions are not sufficient in general, however, and the complete stability classification on the critical spectral boundary remains open.

For simultaneously diagonalizable coefficient matrices, the problem reduces to scalar systems along the common eigendirections, leading to a complete classification in terms of the quantities
\[ \sum_{k=1}^m |a_{k,j}|. \]
This also yields simple sufficient conditions involving the individual spectral radii $\rho(A_k)$. The examples given in the paper show that such conditions do not extend to arbitrary matrix families: even when
\[ \sum_{k=1}^m\rho(A_k)<1, \]
a critical system can be unstable if a toral pencil has a defective unimodular eigenvalue.

Several natural questions remain for future work. One concerns whether attractivity alone implies Lyapunov stability in the bounded-boundary/Pringsheim framework considered here.  Establishing this implication would show that attractivity and asymptotic stability coincide for this class of systems. 

A second direction is the stability theory for general commuting coefficient families. The simultaneously diagonalizable case treated here suggests that the joint spectral structure should govern the strict stability regimes, while nilpotent parts of the critical joint blocks should determine the behavior on the spectral boundary. A complete characterization for arbitrary commuting families would extend the scalar reduction obtained in Section~6 and provide a natural next step toward understanding the critical regime.

More generally, the critical case for unrestricted matrix families remains the principal open problem. The results of this paper identify important necessary conditions and show that Jordan structure plays a central role, but they also demonstrate that pointwise spectral information about the pencil $P(z)$ does not by itself determine Lyapunov stability. A sharper structural understanding of this boundary regime therefore appears to be essential for a complete multidimensional stability theory.

\section*{Acknowledgments}

The authors would like to thank the anonymous referees for their careful reading of the manuscript and for several helpful comments and suggestions that improved the presentation of the paper.

\section*{Statements and Declarations}

\noindent\textbf{Funding.}
The authors received no funding for this work.

\medskip
\noindent\textbf{Competing interests.}
The authors declare that they have no competing interests.

\medskip
\noindent\textbf{Data availability.}
No datasets were generated or analyzed during the
current study. Data sharing is not applicable.

\end{document}